\documentclass[12pt,a4paper]{amsart}      

\usepackage{amsmath,amsthm,amssymb,latexsym,a4wide,tikz,multicol,tikz-cd}
\usepackage{arydshln,multirow}
\usepackage{tikz-qtree,tikz-qtree-compat}
\usepackage{mathtools,stmaryrd}
\usepackage{tikz}
\tikzset{font=\small}
\usepackage{enumerate}
\usetikzlibrary{matrix,arrows}
\usetikzlibrary{positioning}
\usetikzlibrary{cd}
\usepackage[colorlinks]{hyperref}

\usepackage{mathrsfs}

\newtheorem{theorem}{Theorem}[section]
\newtheorem{lemma}[theorem]{Lemma}

\newtheorem{corollary}[theorem]{Corollary}
\newtheorem{proposition}[theorem]{Proposition}

\theoremstyle{definition}
\newtheorem*{definition*}{Definition}

\newtheorem{question}[theorem]{Question}
\newtheorem{problem}[theorem]{Problem}

\newcommand{\N}{\mathbb N}
\newcommand{\F}{\mathcal{F}}

\newcommand{\IZ}{\mathbb{Z}}

\newcommand{\Sch}{\operatorname{Sch}}
\newcommand{\Schf}{\Sch^{\operatorname{fin}}}
\newcommand{\Schp}{\Sch^{\operatorname{p}}}

\title{Combinatorics of Schur ultrafilters}
\author{Serhii Bardyla}
\address{Serhii Bardyla: University of Vienna, Institute of Mathematics, Vienna, Austria.}
\email{sbardyla@gmail.com}

\makeatletter
\@namedef{subjclassname@2020}{%
  \textup{2020} Mathematics Subject Classification}
\makeatother


\thanks{The research of the author was funded in whole by the Austrian Science Fund FWF [10.55776/ESP399].}
\begin{document}

\begin{abstract}
     In this paper, we provide a combinatorial characterization of the elements of Schur ultrafilters on countable commutative groups. Using this characterization, we construct a Schur ultrafilter on $\mathbb Z$ that is not infinitary Schur. Moreover, assuming the Continuum Hypothesis, we establish the existence of Schur P-points in $\beta(\mathbb Z)$. On the other hand, it is consistent with ZFC that there exist P-points in $\beta(\mathbb Z)$, but none of them are Schur. Also, we extend the result of Fern\'{a}ndez-Bret\'{o}n, Navarro-Castillo, and Soria-Rojas by showing that no Schur ultrafilter on $\mathbb Z$ is a Q-point. 
\end{abstract}

\maketitle

\section{Introduction and main results}
In this paper, $\mathbb Z$ denotes the additive group of integers, and $\mathbb N$ denotes the set of positive integers.
A family $\F$ of subsets of a set $X$ is called a {\em filter on} $X$ if
\begin{enumerate}[\rm(i)]
\item $\emptyset\notin \F$;
\item $A\cap B\in\F$ for every $A,B\in\F$;
\item if $A\in \F$ and $B\supseteq A$, then $B\in\F$.
\end{enumerate}
A filter $\F$ on a set $X$ is called an {\em ultrafilter} if for each $A\subseteq X$ either $A\in\F$ or $X\setminus A\in \F$. A family $\mathcal B\subseteq \F$ is called a {\em base} of the filter $\F$ if for each $F\in\F$ there exists $B\in\mathcal B$ such that $B\subseteq F$. A filter $\F$ is called {\em free} if $\bigcap \F=\emptyset$.


The Stone-\v{C}ech compactification $\beta (X)$ of a discrete space $X$ is the set of all ultrafilters on $X$ endowed with a topology $\tau$ generated by the base $\mathcal B=\{\langle A\rangle: A\subseteq X\}$, where $$\langle A\rangle=\{u\in\beta (X): A\in u\}.$$ Recall that each element $x\in X$ is identified with the principal ultrafilter $\{A\subseteq X: x\in A\}$. If $S$ is a discrete semigroup, then the semigroup operation on $S$ can be canonically lifted to a semigroup operation on $\beta (S)$ as follows: if $u,v\in\beta (S)$, then $uv$ is the ultrafilter with the base consisting of the sets $\bigcup_{x\in U}xV_x$, where $U\in u$ and $\{V_x:x\in U\}\subseteq v$ are arbitrary. 
Equivalently, $$A\in uv \iff \{s\in S : \{x\in S: sx\in A\}\in v\} \in u.$$  
By~\cite[Theorem~4.1]{HS}, the defined above semigroup operation on $\beta (S)$ is unique among those extending the operation of $S$ and satisfying the following two natural conditions:
\begin{itemize}
    \item[(i)] for each $u\in\beta (S)$ the right shift $\rho_u:\beta(S)\rightarrow \beta(S)$, $x\mapsto xu$ is continuous;
    \item[(ii)] for each $s\in S$ the left shift $\lambda_s:\beta(S)\rightarrow \beta(S)$, $x\mapsto sx$ is continuous.
\end{itemize}
For more about the algebraic structure of $\beta (S)$ see the monograph~\cite{HS} and references therein.

\begin{definition*}
 An ultrafilter $u$ on a semigroup $S$ is called 
 \begin{enumerate}[\rm(i)]
     \item {\em Schur} if for each $U\in u$ there exist $a,b\in U$ such that $ab\in U$;
     \item {\em infinitary Schur} if for each $U\in u$ there exist $a\in U$ and an infinite subset $B\subseteq U$ such that $aB\subseteq U$.
 \end{enumerate}
\end{definition*}

The classical Ellis's theorem~\cite[Theorem 2.5]{HS} implies the existence of free idempotent ultrafilters on every infinite semigroup. In other words, for each infinite semigroup $S$ there exists a free ultrafilter $u\in\beta(S)$ such that $uu=u$. 
It is straightforward to check that an ultrafilter $u$ on a semigroup $S$ is idempotent if and only if for each $U\in u$ there exists $x\in U$ and $V_x\in u$ such that $V_x\subseteq U$ and $xV_x\subseteq U$.
Thus, for a free ultrafilter $u$ on a semigroup the following implications hold:
$$u\hbox{ is idempotent }\Longrightarrow u\hbox{ is infinitary Schur }\Longrightarrow u\hbox{ is Schur}.$$

Schur ultrafilters were introduced by Protasov~\cite{P} to construct peculiar group topologies. Recently~\cite{BZl, Zl} it was shown that these ultrafilters play a crucial role in a description of the Bohr compactification of topological groups, as well as they determine the automatic continuity of group operations in so-called chart groups.

\begin{definition*}
   For a discrete group $G$ let 
   \begin{enumerate}[\rm(i)]
       \item $\Sch(G)$ be the subspace of $\beta(G)$ consisting of all free Schur ultrafilters;
       \item $\Sch^{\infty}(G)$ be the subspace of $\beta(G)$ consisting of all free infinitary Schur ultrafilters;
       \item $\Schf(G)$ be the subspace of $\beta(G)$ consisting of all free Schur ultrafilters which are not infinitary Schur.
   \end{enumerate}
\end{definition*}
By~\cite[Lemma~5.1]{P} and~\cite[Proposition 3.7, Lemma 3.18, Proposition 3.19]{BZl}, for a commutative group $G$ we have that $\Sch(G)$ is closed nowhere dense subsemigroup of $\beta(G)$ and $\Sch^{\infty}(G)$ is a closed left ideal in $\Sch(G)$. In particular, $\Sch^{\infty}(\mathbb Z)$ is a subsemigroup of $\beta(\mathbb Z)$. On the other hand, it is known (see \cite[Exercise 6.1.4]{HS}) that the set of all idempotent ultrafilters on $\mathbb Z$ is not a subsemigroup of $\beta(\mathbb Z)$. Thus there exist infinitary Schur ultrafilters on $\mathbb Z$ that are not idempotent.

The following problem was posed in~\cite[Problem 3.15]{BZl}
\begin{problem}[Bardyla, Zlato\v{s}]\label{prob1}
Does there exist a free Schur ultrafilter on a countable group $G$ that is not infinitary Schur?     
\end{problem}

The first principal result of this paper is the following solution of Problem~\ref{prob1}.

\begin{theorem}\label{dense}
 The set $\Schf(\mathbb Z)$ is an open dense subset of $\Sch(\mathbb Z)$.   
\end{theorem}

\begin{definition*}
    A free ultrafilter $u$ on a countable set $X$ is called 
\begin{enumerate}[\rm (i)]
    \item {\em weak P-point} if $u$ is not in the closure of any countable subset of $\beta X\setminus (X\cup \{u\})$;
    \item {\em P-point} if for every countable subset $\{U_n:n\in\N\}$ of $u$, there exists $U\in u$ such that $U\subseteq^* U_n$ for all $n\in\N$; 
    \item {\em Q-point} if for every partition $\{C_n:n\in\N\}$ of $X$ such that $|C_n|<\aleph_0$ for all $n\in\N$, there exists $U\in u$ satisfying $|U\cap C_n|\leq 1$ for all $n\in\N$; 
    \item {\em selective} if $u$ is a P-point and a Q-point simultaneously. 
\end{enumerate}
\end{definition*}
It is known that each P-point is a weak P-point.
Kunen~\cite{Kunen} showed that ZFC implies the existence of weak P-points in $\beta(\mathbb Z)$. However, the existence of P-points, Q-points, and selective ultrafilters on $\mathbb Z$ is consistent with ZFC, but cannot be derived from ZFC alone (see~\cite{Miller, Shelah}). Notice that a free idempotent ultrafilter $u$ on a countable group $G$ cannot be a weak P-point, as $u$ is an accumulation point of the countable family $\{xu: x\in G\setminus\{1_G\}\}\subseteq \beta G\setminus (G\cup\{u\})$. 

The following result was proven in~\cite[Propositions 3.12 and 3.13]{BZl} 
\begin{proposition}[Bardyla, Zlato\v{s}]\label{old}
For a countable group $G$ the following assertions hold: 
\begin{enumerate}[\rm (i)]
    \item if $u\in \Sch^{\infty}(G)$, then $u$ is not a P-point;
    \item if $u\in \Sch(G)$ and $G$ is commutative, then $u$ is not selective.
\end{enumerate}
\end{proposition}

Proposition~\ref{old} motivates the following problem, posed in~\cite[Problem 3.16]{BZl}.

\begin{problem}[Bardyla, Zlato\v{s}]\label{prob2}
Is the existence of a free Schur ultrafilter on $\mathbb Z$ that is a P-point consistent with ZFC?    
\end{problem}

By $\Sch^{\operatorname{p}}(\IZ)$ we denote the set of all free Schur ultrafilters on $\IZ$ which are P-points. The second principal result of this paper solves Problem~\ref{prob2} under the Continuum Hypothesis (shortly, CH).


\begin{theorem}\label{pp}
{\em (CH)} The set $\Schp(\IZ)$ is dense in $\Sch(\IZ)$. 
\end{theorem}

Shelah constructed a model of ZFC where $\beta (\mathbb Z)$ contains P-points but all of them are selective (see~\cite[Section XVIII.4]{Shelahbook}). Then, taking into account Proposition~\ref{old}, we get the following.

\begin{corollary}
It is consistent with ZFC that there exist P-points in $\beta(\mathbb Z)$, but no Schur P-points.   
\end{corollary}

In this paper, we adopt multiplicative notation for the semigroup operation on $\beta(\mathbb Z)$. In particular, for an ultrafilter $u \in \beta(\mathbb Z)$ and $n\in \mathbb N$, we denote by $u^{n}$ the product of $n$ copies of $u$.
The following result was established in~\cite{Q-points}.

\begin{theorem}[Fern\'{a}ndez-Bret\'{o}n, Navarro-Castillo, and Soria-Rojas]\label{newww}
If u is a Q-point on $\mathbb{Z}$, then $u^n$ is not an idempotent ultrafilter. In particular, no idempotent ultrafilter on $\mathbb{Z}$ can be a Q-point.  
\end{theorem}

The third principal result of this paper is the following extension of Theorems~\ref{old} and~\ref{newww}.

\begin{theorem}\label{Q}
If u is a Schur ultrafilter on $\mathbb{Z}$, then u is not a Q-point.    
\end{theorem}

We left the following natural question open.

\begin{question}
Does there exist an infinitary Schur ultrafilter on $\mathbb Z$ that is a weak P-point?    
\end{question}

A surjective function $c: X\rightarrow \{1,\ldots, r\}$ is called an {\em $r$-coloring} of a set $X$. 

\begin{definition*}
Let $G$ be a countable commutative group. An infinite subset $\{x_n:n\in\N\}\subseteq G$ is called {\em Schur} if for every $r,k\in\N$ there exists a minimal number $S(r,k)\in\N$ such that for every $r$-coloring of $X=\{x_n: n\leq S(r,k)\}$ there exist $a\in X$ and $\{b_1,\ldots,b_k\}\subseteq X$ such that $\{a\}\cup \{b_1,\ldots,b_k\}\cup \{a+b_i: i\leq k\}$ is a monochromatic subset of $X$.
\end{definition*}
It is easy to see that the Schur property of a subset $A$ does not depend on the enumeration of $A$. In contrast, the numbers $S(r,k)$ do depend on the enumeration of $A$. The following combinatorial description of elements of free Schur ultrafilters on countable commutative groups, which is crucial for the proofs of Theorems~\ref{dense} and~\ref{pp}, is the fourth principal result of this paper.

\begin{theorem}\label{Schursubset}
Let $G$ be a countable commutative group. A subset $A\subseteq G$ belongs to some free Schur ultrafilter on $G$ if and only if $A$ is Schur.     
\end{theorem}

It is clear that the set $\N$ of positive integers with its usual increasing enumeration is a Schur subset of $\IZ$. Thus, for each $r,k\in\N$ the positive integer $S(r,k)$ is well-defined. 
We refer to these integers as {\em two dimensional Schur numbers}. Recall that the classical Schur numbers $S(r)$, $r\in\N$ are defined as follows: $S(r)$ is the minimal positive integer such that for every $r$-coloring of the set $\{n\in\N: n\leq S(r)\}$ there exist $a,b\in X$ such that the set $\{a,b,a+b\}$ is monochromatic. In our terminology, we have $S(r)=S(r,1)$ for each $r\in\N$. Despite its relatively simple definition, the computational complexity of the classical Schur numbers is extremely high. In particular, the equality $S(5)=161$
was established as recently as 2017 by Heule~\cite{H} using computers. Moreover, their computation required 2 petabytes of space, see the news article~\cite{news}. This leads to a fairly general and complicated question about the values of two-dimensional Schur numbers, which we leave open.


\section{Proofs of main results} 
The following result appears in~~\cite[Corollary 3.11]{BZl}.

\begin{proposition}[Bardyla, Zlato\v{s}]\label{nice}
Let $u$ be a Schur ultrafilter on a commutative group $G$. Then for each $U\in u$ and $n\in\N$ there exists $x\in U$ such that $|\{y\in U:x+y\in U\}|\geq n$.    
\end{proposition}

\begin{proposition}\label{fin}
Let $u$ be a free Schur ultrafilter on a countable commutative group $G$. Then every $U\in u$ is a Schur set.  
\end{proposition}

\begin{proof}
Fix an enumeration $U=\{x_n: n\in \N\}$. Assuming the contrary, there exist $r,k\in\N$ such that for every $p\in\N\setminus\{1,\ldots, r-1\}$ there exists an $r$-coloring $\phi_p$ of $\{x_n: n\leq p\}$ with the following property: for each monochromatic subset $B$ of $\{x_n: n\leq p\}$, for each $y\in B$ we have $|\{z\in B: y+z\in B\}|<k$. The Pigeonhole principle yields a color $c_1\in \{1,\ldots, r\}$ such that the set $M[1]=\{p\in\N\setminus\{1,\ldots, r-1\}: \phi_p(x_1)=c_1\}$ is infinite. Similarly, there exists a color $c_{2}\in \{1,\ldots, r\}$ such that the set $M[2]=\{p\in M[1]: \phi_p(x_{2})=c_{2}\}$ is infinite. Proceeding in this way, we obtain a decreasing family $\{M[n]: n\in\N\}$ of infinite subsets of $\N$, and a finite coloring $\psi$ of $U$ that colors $x_n$ into $c_n$. Since $u$ is an ultrafilter, there exists $q\leq r$ such that $\psi^{-1}(q)\in u$. By Proposition~\ref{nice}, there exist $a\in \psi^{-1}(q)$ and $B\subseteq \psi^{-1}(q)$ such that $|B|=k$ and $a+B\subseteq \psi^{-1}(q)$. Let $p'$ be any positive integer such that for each $x_n\in \{a\}\cup B\cup (a+B)$ we have $n\leq p'$.
Fix any $d\in M[p']$ and notice that the colorings $\phi_d$ and $\psi$ have the same traces on $\{x_n: n\leq p'\}$. But then the monochromatic subset $\{a\}\cup B\cup (a+B)\subseteq \{x_n: n\leq p'\}$ contradicts the choice of $\phi_d$.
\end{proof}


The proof of the following lemma is straightforward and thus is left to the reader.

\begin{lemma}\label{cofin}
 If $A$ is a Schur subset of a countable commutative group $G$, then every subset $B\subseteq G$ such that $A\subseteq^* B$ is Schur.   
\end{lemma}

Recall that a family $\mathcal R$ of nonempty subsets of a set $X$ is called {\em partition regular} if for each finite family $\mathcal A$ of subsets of $X$ such that $\bigcup\mathcal A\in \mathcal R$ there exists $A\in\mathcal A$ and $B\in\mathcal R$ such that $B\subseteq A$ (see \cite[Chapter 3.1]{HS} for more information).

The following result implies that the family of all Schur subsets of a countable commutative group $G$ is partition regular.

\begin{proposition}\label{partition}
Let $A$ be a Schur subset of a countable commutative group $G$. Then for every partition $A=X\sqcup Y$ either $X$ is Schur or $Y$ is Schur.    
\end{proposition}

\begin{proof}
Let $A=\{a(n):n\in\N\}$ be an enumeration of $A$.
Seeking a contradiction, assume that there exists a partition $A=X\sqcup Y$ such that neither $X$ nor $Y$ is Schur. By Lemma~\ref{cofin}, without loss of generality we can assume that both $X$ and $Y$ are infinite. 
Let $$X=\{a(n_l):l\in\N\}\qquad \hbox{ and }\qquad Y=\{a(m_l):l\in\N\},$$ where the sequences $\{n_l: k\in\N\}$ and $\{m_l:l\in\N\}$ are increasing. 
By the assumption, there exist $r_1,k_1\in\N$ such that for every $p\in\N\setminus\{1,\ldots, r_1-1\}$ there exists an $r_1$-coloring $\phi^X_p$ of $\{a(n_k): k\leq p\}$ with the following property: for each monochromatic subset $B$ of $\{a(n_k): k\leq p\}$, for each $b\in B$ we have $|\{z\in B: b+z\in B\}|<k_1$. Similarly, there exist $r_2,k_2\in\N$ such that for every $p\in\N\setminus\{1,\ldots, r_2-1\}$ there exists an $r_2$-coloring $\phi^Y_p$ of $\{a(m_k): k\leq p\}$ with the following property:  for each monochromatic subset $B$ of $\{a(m_k): k\leq p\}$, for each $b\in B$ we have $|\{z\in B: b+z\in B\}|<k_2$.

Let $r=r_1+r_2$ and $k=\max\{k_1,k_2\}$. Since $A$ is a Schur set, Proposition~\ref{fin} yields a number $S(r,k)\in\N$ such that for every $r$-coloring of the set $T=\{a(n): n\leq S(r,k)\}$ there exist $a\in T$ and $B\subseteq T$ such that $|B|=k$ and the set $\{a\}\cup B\cup (a+B)\subseteq T$ is monochromatic. Put $$f=|X\cap T|\qquad \hbox{ and }\qquad g=|Y\cap T|.$$ By the choice of enumerations of $X$ and $Y$ we get 
$$T= \{a(n_k): k\leq f\}\cup  \{a(m_k): k\leq g\}.$$
Define an $r$-coloring $\psi$ of $T$ as follows:
$$
\psi(a)=\begin{cases} 
\phi_f^X(a), &\hbox{if } a\in \{a(n_k): k\leq f\};\\
\phi_g^Y(a)+r_1, &\hbox{if } a\in \{a(m_k): k\leq g\}.
\end{cases}
$$ 

By the choice of $S(r,k)$, there exist $a\in T$ and $B\subseteq T$ such that $|B|=k$ and the set $\{a\}\cup B\cup (a+B)\subseteq T$ is monochromatic with respect to the coloring $\psi$. By the definition of the coloring $\psi$, either $\{a\}\cup B\cup (a+B)\subseteq X$ or  $\{a\}\cup B\cup (a+B)\subseteq Y$. In the former case the set $\{a\}\cup B\cup (a+B)$ is monochromatic with respect to the coloring $\phi^X_f$, whereas in the latter case the set $\{a\}\cup B\cup (a+B)$ is monochromatic with respect to the coloring $\phi^Y_g$. Taking into account the definition of $k$, we obtain a contradiction with the choice of $\phi^X_f$ and $\phi^Y_g$.    
\end{proof}

We are in a position to prove Theorem~\ref{Schursubset}. We need to show that a subset $A$ of a countable commutative group $G$ belongs to some free Schur ultrafilter on $G$ if and only if $A$ is Schur.

\begin{proof}[{\bf Proof of Theorem~\ref{Schursubset}}]
The implication ($\Rightarrow$) follows from Proposition~\ref{fin}.

The implication ($\Leftarrow$) follows from Proposition~\ref{partition} and~\cite[Theorem 3.11]{HS} applied to the partition regular family of all Schur subsets of $G$.
\end{proof}

We are in a position to prove Theorem~\ref{dense}. We need to show that $\Schf(\mathbb Z)$ is an open dense subset of $\Sch(\mathbb Z)$.

\begin{proof}[{\bf Proof of Theorem~\ref{dense}}]
 Fix any free Schur ultrafilter $u$ on $\IZ$ and its element $U$. Without loss of generality we can assume that $U\subseteq \N$. Consider the increasing enumeration $U=\{x_n:n\in\N\}$. By Theorem~\ref{Schursubset}, to show that $\Schf(\mathbb Z)$ is dense in $\Sch(\mathbb Z)$, it suffices to construct a Schur subset $W\subseteq U$ such that for each $a\in W$ the set $\{b\in W: a+b\in W\}$ is finite. By Proposition~\ref{fin}, $U$ is a Schur set. Hence the number $p_1=S(1,1)$ is well defined for the given enumeration of $U$. 
 Let $B_1=\{x_n: n\leq p_1\}$. Since the ultrafilter $u$ is free, we have $U_1=\{x_n:n> 2p_1\}\in u$. Consider a finite partition of the set $U_1$ into the sets $A_i^1$, $0\leq i\leq x_{p_1}$ defined as follows: for every $a\in U_1$, set $a\in A_i^1$ if and only if the remainder of the division of $a$ by $x_{p_1}+1$ equals $i$. Since $u$ is an ultrafilter and $U_1\in u$, there exists $0\leq i\leq x_{p_1}$ such that $A_i^1\in u$. Fix the increasing enumeration $A_i^1=\{a^1_n: n\in \N\}$. By Proposition~\ref{fin}, the set $A^1_i$ is Schur. Thus the number $p_2=S(2,2)$ is well defined for the given enumeration of $A^1_i$. 
 Let $B_2=\{a^1_n: n\leq p_2\}$.
 Note that $(B_1+B_1)\cap B_2=\emptyset$, as $U_1\cap \{x_n: n\leq 2p_1\}=\emptyset$ and $B_2\subseteq A^1_i\subseteq U_1$. Recall that for each $a\in B_1$ we have $a< x_{p_1}+1$. Then for every $y\in B_2$ the remainder of the division of $a+y$ by $x_{p_1}+1$ does not equal $i$, which implies $a+y\notin B_2$. Thus $(B_1+B_2)\cap B_2=\emptyset$. Also, it is clear that $(B_2+B_2)\cap B_1=\emptyset$. Hence for any $i,j,m\in\{1,2\}$ we have $$(B_i+B_j)\cap B_m\neq \emptyset \quad\Longleftrightarrow \quad i=j=m.$$ By the choice of $B_2$, for each $2$-coloring of $B_1\cup B_2$ there exist $b\in B_2$ and $\{c,d\}\subseteq B_2$ such that $\{b,c,d,b+c,b+d\}$ are monochromatic.

 Assume that we have already constructed a family of finite pairwise disjoint subsets $\{B_j: j\leq k\}$ of $U=\{x_n:n\in\N\}$,  such that 
 \begin{enumerate}
     \item for all $i,j,m\in\{1,\ldots, k\}$ we have $(B_i+B_j)\cap B_m\neq \emptyset$ if and only if $i=j=m$;
     \item for every $j\leq k$, for each $j$-coloring of $\bigcup_{i\leq j}B_i$ there exist $b\in B_j$ and a subset $C\subseteq B_j$ of cardinality $j$ such that the set $\{b\}\cup C\cup(b+C)\subseteq B_j$ is monochromatic.
 \end{enumerate}
 Let $q=\max\{n\in\N: x_n\in \bigcup_{j\leq k}B_j\}$. Since $u$ is a free Schur ultrafilter, we have that the set $U_{k}=\{x_n: n>2q\}\subseteq U$ is a Schur set that belongs to $u$. Consider a finite partition of the set $U_k$ into the sets $A_i^k$, $0\leq i\leq x_{q}$ defined as follows: for every $a\in U_k$, we have $a\in A_i^k$ if and only if the remainder of the division of $a$ by $x_{q}+1$ equals $i$. Since $U_k\in u$ and $u$ is an ultrafilter, there exists $0\leq i\leq x_{q}$ such that $A_i^k\in u$. Fix the increasing enumeration $A_i^k=\{a^k_n: n\in \N\}$. Since the set $A^k_i$ is Schur, the number $p_{k+1}=S(k+1,k+1)$ is well defined for the given enumeration of $A^k_i$. Put $B_{k+1}=\{a^k_n: n\leq p_{k+1}\}$. Similarly as in the case $k=2$, one can check that the family $\{B_j: j\leq k+1\}$ satisfies the inductive assumptions (1) and (2) for $k+1$.

Upon completing the recursion we get a family $\{B_n:n\in \N\}$ consisting of pairwise disjoint finite subsets of $U$ which satisfies the following two conditions:
 \begin{enumerate}
     \item $(B_i+B_j)\cap B_m\neq \emptyset$ if and only if $i=j=m$;
     \item for every $n\in\N$, for each $n$-coloring of $\bigcup_{i\leq n}B_i$ there exist $b\in B_n$ and a subset $C\subseteq B_n$ of cardinality $n$ such that the set $\{b\}\cup C\cup(b+C)\subseteq B_n$ is monochromatic.
 \end{enumerate}

Put $W=\bigcup_{n\in\N}B_n$. It is clear that $W\subseteq U$.
 Condition (2) ensures that $W$ is a Schur set. Condition (1) implies that for every $a\in B_n\subseteq W$ the set 
 $\{b\in W: a+b\in W\}$ is finite, as it is contained in $B_n$. Hence $\Schf(\mathbb Z)$ is dense in $\Sch(\mathbb Z)$.
In order to show that $\Schf(\mathbb Z)$ is open in $\Sch(\IZ)$ fix any $v\in \Schf(\IZ)$. Then there exists an element $V\in v$ such that for each $a\in V$ the set $\{b\in V: a+b\in V\}$ is finite. It is easy to see that $\langle V\rangle\cap \Sch(\IZ)\subseteq \Schf(\IZ)$, witnessing that the set $\Schf(\IZ)$ is open in $\Sch(\IZ)$.
\end{proof}

We are in a position to prove Theorem~\ref{pp}. We need to show that under CH the set $\Schp(\IZ)$ (consisting of all Schur P-points) is dense in $\Sch(\IZ)$.

\begin{proof}[{\bf Proof of Theorem~\ref{pp}}]
Fix a Schur subset $A\subseteq \IZ$. Without loss of generality we can assume that $A\subseteq \N$. By Theorem~\ref{Schursubset}, it suffices to construct a Schur $P$-point $u$ such that $A\in u$. Using CH, enumerate the set of all infinite subsets of $\IZ$ as $\{A_\alpha: \alpha\in \operatorname{Succ}(\aleph_1)\}$, where $\operatorname{Succ}(\aleph_1)$ is the set of all successor ordinals in $\aleph_1$. Also assume that $A=A_0$. Let $U_0=A_0=A$. Assume that for all $\xi< \gamma<\aleph_1$ we constructed Schur subsets $U_\xi$ of $\IZ$ which satisfy the following two conditions:
 \begin{enumerate}
     \item $U_\alpha\subseteq^* U_\beta$ for all $\beta\leq \alpha<\gamma$;
     \item for every successor ordinal $\xi\in\gamma$ we have either $U_\xi\subseteq A_\xi$ or $U_\xi\subseteq G\setminus A_\xi$.
 \end{enumerate}
 
If $\gamma=\delta+1$, then Lemma~\ref{partition} implies that either $U_\delta\cap A_\gamma$, or $U_\delta\cap (G\setminus A_\gamma)$ is a Schur set. In the former case put $U_\gamma=U_\delta\cap A_\gamma$, whereas in the latter case set $U_\gamma=U_\delta\cap (G\setminus A_\gamma)$. 

Assume that $\gamma$ is a limit ordinal. There exists a cofinal increasing subset $\{\mu(n):n\in\N\}$ in $\gamma$. For each $n\in\N$ let $V_n=\bigcap_{i\leq n}U_{\mu(i)}$.  
Note that $U_{\mu(n)}\subseteq^*V_n$ for each $n\in\N$. Lemma~\ref{cofin} implies that the sets $V_n$, $n\in\N$ are Schur. For each $n\in\N$ fix the increasing enumeration $V_{n}=\{x^n_{k}: k\in\N\}$. 
Then for each $n\in \N$ the number $p_n=S(n,n)$ is well defined for the given enumeration of $V_n$. Let $B_n=\{x^n_k: k\leq p_n\}$ for every $n\in\N$.  Note that the family $\{B_n: n\in\N\}$ satisfies the following conditions:
 \begin{enumerate}[\rm(i)]
 \item $B_n\subseteq V_n\subseteq U_{\mu(n)}$ for all $n\in\N$;
     \item for every $n\in\N$, for each $\leq n$-coloring of $\bigcup_{i\in\N}B_i$ there exist $b\in B_n$ and a subset $C\subseteq B_n$ of cardinality $n$ such that the set $\{b\}\cup C\cup(b+C)\subseteq B_n$ is monochromatic.
 \end{enumerate}

Put $U_\gamma=\bigcup_{n\in\N}B_n$. 
 Condition (ii) ensures that $U_\gamma$ is a Schur set. Since $V_{n+1}\subseteq V_n$ for all $n\in\N$ and the sets $B_n$, $n\in\N$ are finite, condition (i) implies that $U_\gamma\subseteq^* V_n\subseteq U_{\mu(n)}$ for all $n\in\N$. Taking into account that the sequence $\{\mu(n):n\in\N\}$ is cofinal in $\gamma$, the inductive assumption (1) implies that $U_\gamma\subseteq^* U_\alpha$ for each $\alpha\leq \gamma$. 

 Upon completing the recursion, we obtain a family $\mathcal{U}=\{U_\alpha:\alpha\in\aleph_1\}$ which is well-ordered by the relation $\supseteq^*$. Let us verify that the filter $u$ on $\mathbb Z$ generated by the base $\mathcal U$ is a Schur P-point. Fix any infinite subset $X$ of $\IZ$. There exists a successor ordinal $\alpha\in\aleph_1$ such that $X=A_\alpha$. The set $U_\alpha\in u$ witnesses that either $A_\alpha\in u$, or $\IZ\setminus A_\alpha\in u$. Hence $u$ is an ultrafilter. Since the base $\mathcal U$ of $u$ consists of Schur sets, Lemma~\ref{cofin} implies that $u$ is a Schur ultrafilter. Finally, fix a countable family $\{X_n:n\in\N\}\subseteq u$. Since $\mathcal U$ is a base of $u$, for each $n\in\N$ there exists $\xi(n)\in\aleph_1$ such that $U_{\xi(n)}\subseteq X_n$. Put $\xi=\sup\{\xi(n):n\in\N\}$ and note that $U_\xi\subseteq^* X_n$ for all $n\in\N$. Hence $u$ is a P-point. 
\end{proof}

It remains to prove Theorem~\ref{Q}, which states that no Schur ultrafilter on $\mathbb Z$ is a Q-point.

\begin{proof}[{\bf Proof of Theorem~\ref{Q}}]
Seeking a contradiction, assume that $u$ is a Schur ultrafilter on $\mathbb Z$ which is a Q-point. We lose no generality assuming that $\mathbb N\in u$. Recursively define the integers $b_n$, $n\in\N$ as follows: put $b_0=0$ and $b_{n+1}=3b_n+1$ for all $n\in\N$. For every $n\in\N$ let $$I_n=\{x\in\N: b_{n-1}+1\leq x\leq b_n\}.$$  
It is clear that $\{I_n:n\in\N\}$ is a partition of $\N\in u$. Since $u$ is a Q-point, there exists $U\in u$ such that $|U\cap I_n|\leq 1$ for all $n\in\N$. Let $A=\{n\in\N: |U\cap I_n|=1\}$. For each $n\in A$ let $U\cap I_n=\{x_n\}$. In particular, $U=\{x_n: n\in A\}$. Let $$A_1=\{n\in A: x_n\leq 2b_{n-1}\} \quad \hbox{ and }\quad A_2=\{n\in A: x_n>2b_{n-1}\}.$$
Since $A_1\cup A_2=A$ and $u$ is an ultrafilter, either $$U_1=\{x_n: n\in A_1\}\in u \quad\hbox{ or }\quad U_2=\{x_n: n\in A_2\}\in u.$$ We are going to show that both these cases lead to a contradiction. 

Assume that $U_1\in u$. Since $u$ is a Schur ultrafilter, by Proposition~\ref{nice}, there exist distinct $i,j\in A_1$ such that $x_i+x_j\in U_1$. We lose no generality assuming that $i<j$. Note that $x_i\leq b_i\leq b_{j-1}$ and $x_j\leq 2b_{j-1}$. It follows that $x_i+x_j\leq 3b_{j-1}<b_j$. Taking into account that $x_i+x_j>x_j$ we get that $\{x_j,x_i+x_j\}\subseteq U\cap I_{j}$, which contradicts the choice of $U$. 

Assume that $U_2\in u$. Since $u$ is a Schur ultrafilter, by Proposition~\ref{nice}, there exist distinct $i,j\in A_1$ such that $x_i+x_j\in U_2$. We lose no generality assuming that $i<j$. Note that $x_i\leq b_i\leq b_{j-1}$, and $2b_{j-1}< x_j \leq b_j= 3b_{j-1}+1$. It follows that $x_i+x_j\leq 4b_{j-1}+1<b_{j+1}$. Since $U\cap I_j=\{x_j\}$, we get that $x_i+x_j\notin I_j$. Then $x_i+x_j=x_{j+1}\in I_{j+1}$. On the other hand, $x_{j+1}=x_i+x_j\leq 4b_{j-1}+1\leq b_j+b_{j-1}\leq 2b_j$. But then $j+1\in A_1$ and as such $x_i+x_j=x_{j+1}\notin U_2$, a contradiction.
\end{proof}


\section*{Acknowledgements}
The author would like to thank Lyubomyr Zdomskyy for a pivotal discussion on this topic.

\end{document}